\documentclass[12pt]{article}
\usepackage{amsmath,amssymb,amsthm,graphicx,url,color}
\usepackage[utf8]{inputenc}
\usepackage[russian,german,french,english]{babel}
\newtheorem{proposition}{Proposition}

\begin{document}
\author{Vladimir Uspenskiy\footnote{Moscow State Lomonosov University}, Alexander Shen\footnote{LIRMM CNRS and University of Montpellier, 161 rue Ada, Montpellier, France. On leave from IITP RAS. Supported by ANR-15-CE40-0016-01 RaCAF grant, \texttt{alexander.shen@lirmm.fr}}}
\title{Algorithms and Geometric Constructions}
\date{}
\maketitle
\begin{abstract}
It is well known that several classical geometry problems (e.g., angle trisection) are unsolvable by compass and straightedge constructions. But what kind of object is proven to be non-existing by usual arguments? These arguments refer to an intuitive idea of a geometric construction as a special kind of an ``algorithm'' using restricted means (straightedge and/or compass). However, the formalization is not obvious, and different descriptions existing in the literature are far from being complete and clear. We discuss the history of this notion and a possible definition in terms of a simple game.
\end{abstract}

\section{Introduction}
The notion of an algorithm as an intuitively clear notion that precedes any formalization, has a rather short history. The first examples of what we now call algorithms were given already by Euclid and al-Khw\^arizm\^\i. But the general idea of an algorithm seems to appear only in 1912 when Borel considered ``les calculus qui peuvent \^etre r\'eellement effectu\'es''\footnote{the computations that can be really performed} and emphasized: ``Je laisse intentionnellement de c\^ot\'e le plus ou moins grande longeur pratique des op\'erations; l'essentiel est que chaqune de ces op\'erations soit ex\'ecutable en un temps fini, par une m\'ethode s\^ure et sans ambigu\"\i te''\footnote{I intentionally put aside the question of bigger or smaller practical length of the operation; it is important only that each of the operations can be performed in a finite time by a clear and unambiguous method''}~\cite[p.~162]{borel1912}. The formal definition of a representative class of algorithms was given in 1930s (in the classical works of G\"odel, Church, Kleene, Turing, Post and others); the famous Church--Turing thesis claims that the class of algorithms provided by these formal definitions is representative (for each algorithm in the intuitive sense there exists an equivalent algorithm in the sense of the formal definition). Later the appearance of computers and programming languages made natural to consider computer programs as a representative class of algorithms (ignoring quantitative restrictions imposed by the finite size of computer memory). Finally, one could say that the abundance of computers (and programmers) will soon make people who have understood the intuitive idea of algorithm \emph{before} learning a programming language an extinct species. If  (when) this happens, the intuitive notion of an algorithm would disappear thus reducing the Church--Turing thesis to a definition.

In this paper we look at the history of another related notion: the notion of a \emph{geometric construction}. One may consider geometric constructions as a special type of algorithms that deal with geometric objects. Euclid provided many examples of geometric constructions by compass and straightedge (ruler); later these constructions became a standard topic for high school geometry exercises. Several classical problems (angle trisection, doubling the square, squaring the circle) were posed and remained unsolved since ancient times (though solutions that involve more advanced instruments than compass and straightedge were suggested). These problems were proved to be unsolvable in 19th century. One would expect that the proof of unsolvability assumes as a prerequisite a rigorously defined notion of a ``solution'' that does not exist. Recall that the first undecidability proofs could appear only after an exact definition of an algorithm was given.

However, historically this was not the case and the impossibility proofs appeared without an exact definition of a ``geometric construction''. These proofs used the algebraic approach: For example, to show that the cube cannot be doubled, one proves that $\sqrt[3]{2}$ cannot be obtained from rationals by arithmetic operations and square roots. The reduction from a geometric question to an algebraic one looks quite obvious and was omitted by Wantzel who first proved the impossibility of angle trisection and cube doubling. As he wrote in~\cite{wantzel1837}, ``pour reconnaitre si la construction d'un probl\`eme de G\'eometrie peut s'effectuer avec la r\`egle et le compas, if faut chercher s'il est possible de faire d\'ependre les racines de l'equation \`a laquelle il conduit de celles d'un syst\`eme d'\'equations du second degr\'e''.\footnote{To find out whether a geometric problem can be solved by straightedge and compass construction, one should find whether it is possible to reduce the task of finding the roots of the corresponding equation to a system of equations of second degree.}  This is said in the first paragraph of the paper and then Wantzel considers only the algebraic question.

Several other interesting results were obtained in 19th century. It was shown that all constructions by compass and straightedge can be performed using the compass only (the \emph{Mohr--Mascheroni theorem}) if we agree that a line is represented by a pair of points on this line. Another famous result from 19th century, the \emph{Poncelet--Steiner theorem}, says that if a circle with its center is given, then the use of compass can be avoided, straightedge is enough. Other sets of tools were also considered, see, e.g., \cite{bieberbach1952,martin1998,hilbert1902}. 

Later geometric construction became a popular topic of recreational mathematics (see, e.g., \cite{rademacher1933,courant1996,manin1963,kac-ulam1968}). In most of the expositions the general notion of a geometric construction is still taken as granted, without a formal definition, even in the nonexistence proofs (e.g., when explaining Hilbert's proof that the center of a circle cannot be found using only a straightedge~\cite{rademacher1933,courant1996,kac-ulam1968}; see below Section~\ref{sec:error} about problems with this argument). Sometimes a definition for some restricted class of geometric construction is given (see, e.g.,~\cite{tao2011}). In~\cite{manin1963} an attempt to provide a formal definition is made, still it remains ambiguous with respect to the use of ``arbitrary points'' (see Section~\ref{sec:arbitrary}).  Baston and Bostock~\cite{baston-bostock1990} observe that the intuitive idea of a ``geometric construction'' has no adequate formal definition and discuss several examples but do not attempt to give a formal definition that is close to the intuitive notion.  It seems that even today people still consider the intuitive notion of a ``geometric construction algorithm'' as clear enough to be used without a formal definition (cf.~\cite{akopyan-fedorov2017}, especially the first \texttt{arxiv} version).

In Section~\ref{sec:derivable} we consider a na\"\i ve approach that identifies constructible points with the so-called ``derivable'' points. Then in Sections~\ref{sec:uniform} and~\ref{sec:arbitrary} we explain why this approach contradicts our intuition. In Section~\ref{sec:game} we suggest a more suitable definition, and finally in Section~\ref{sec:error} we note that the absence of formal definitions has led to incorrect proofs.

\section{Derivable Points and Straight-Line Programs}\label{sec:derivable}

At first it seems that the definition of a geometric construction is straightforward. We have three classes of geometric objects: points, lines and circles. Then we consider some operations that can be performed on these objects. We need to obtain some object (the goal of our construction) applying the allowed operations to given objects. As Tao~\cite{tao2011} puts it, 
\begin{quote}
Formally, one can set up the problem as follows. Define a configuration to be a finite collection $\mathcal{C}$ of points, lines, and circles in the Euclidean plane. Define a construction step to be one of the following operations to enlarge the collection $\mathcal{C}$:

\begin{itemize}
\item
(Straightedge) Given two distinct points $A$, $B$ in $\mathcal{C}$, form the line $\overline{AB}$ that connects $A$ and $B$, and add it to $\mathcal{C}$.

\item
(Compass) Given two distinct points $A$, $B$ in $\mathcal{C}$, and given a third point $O$ in $\mathcal{C}$ (which may or may not equal $A$ or $B$), form the circle with centre $O$ and radius equal to the length $|AB|$ of the line segment joining $A$ and $B$, and add it to $\mathcal{C}$.

\item
(Intersection) Given two distinct curves $\gamma$, $\gamma'$ in $\mathcal{C}$ (thus $\gamma$ is either a line or a circle in $\mathcal{C}$, and similarly for $\gamma'$), select a point $P$ that is common to both $\gamma$ and $\gamma'$ (there are at most two such points), and add it to $\mathcal{C}$.
\end{itemize}

We say that a point, line, or circle is constructible by straightedge and compass from a configuration $\mathcal{C}$ if it can be obtained from $\mathcal{C}$  after applying a finite number of construction steps.
\end{quote}%

We can even try to define the geometric construction algorithm as a straight-line program, a sequence of assignments whose left-hand side is a fresh variable and the right-hand side contains the name of the allowed operation and the names of objects (variables) to which this operation is applied. A definition of this type is given by Huckenbeck~\cite{huckenbeck,huckenbeck2} where this type of programs is called $\textit{GCM}_0$.

Baston and Bostock~\cite{baston-bostock1990} use the name ``derivable'' for objects that can be obtained in this way starting from given objects. In other words, starting with some set of given objects, they consider its closure, i.e., the minimal set of objects that contains the given ones and is closed under allowed operations. The objects that belong to this closure are called \emph{derivable} from the given ones. In these terms, the impossibility of trisecting the angle with the compass and the straightedge can be stated as follows: \emph{for some points $A,B,C$ the trisectors of the angle $BAC$ are not derivable from $\{A,B,C\}$}.

Baston and Bostock note that the intuitive notion of a ``constructible'' point (that they intentionally leave without any definition) may differ from the formal notion of a derivable point in both directions. We discuss the possible differences in the following sections.

\section{Uniformity and Tests}\label{sec:uniform}

There are some problems with the approach based on the derivability. First of all, this appoach is ``non-uniform''. Asking a high school student to construct, say, a center of an inscribed circle of a triangle $ABC$, we expect the solution to be some specific construction that works \emph{for all triangles}, not just the proof that this center is always derivable from $A$, $B$, and $C$. The na\"\i ve approach would be to ask for a straight-line program that computes this center starting from $A$, $B$, and $C$. However, an obvious problem arises: the operation of choosing an intersection point of two curves is non-deterministic (we need to choose one of the two intersection points).
For example, Engeler~\cite{engeler1968} includes two operations in his list of ``operating capabilities'':
\begin{quote}
\begin{description}
 \item[(6)] $P_i:=(\gamma_j,\gamma_k)$, choose an intersection point of two intersecting circles;
 \item[(7)] $P_i=(P_s,\gamma_j,\gamma_k)$, given an intersection point of two circles, find the other intersection point.
\end{description}
\end{quote}
(p.~65; similar operations for lines and circles are also in the list.) In this way the straight-line program is no more deterministic.  We may guarantee only that \emph{some} run of the program produces the required object, or guarantee that the required object is among the objects computed by this program. This is a common situation for classical constructions. For example, the standard construction of the centre of the incircle of a triangle can also produce centres of excircles (the circles outside the triangle that touch one of its sides and the extensions of two other sides).

Bieberbach~\cite[p.~26]{bieberbach1952} describes this problem as follows:
\begin{quote}
\foreignlanguage{german}{%
Noch mag aber ausdrücklich hervorgehoben werden, daß es sich in diesem Paragraphen ebenso wie bei den Poncelet--Steinerschen Konstruktionen stets um ein Konstruieren in der orientierten Ebene handelt. Es soll bei jedem gegebenen und bei jedem konstruierten Punkt feststehen, welches die Vorzeichen seiner Koordinaten sind. Anderenfalls steht nur fest, daß der gesuchte Punkt sich unter den konstruierten befindet.}\footnote{One should say explicitly that the constructions of this paragraph (as well as the Poncelet--Steiner constructions) are performed in the oriented plane. For each given and constructed point it should be determined what the signs of its coordinates are. Otherwise it is only certain that the point in question is among the constructed points.}
\end{quote}
Bieberbach then adds (on p.~151):
\begin{quote}
\foreignlanguage{german}{
\textsc{H.Tietze} hat (l.c.,\S 4) zu dem Ergebnis dieses Paragraph die folgende Bemerkung gemacht: Bei der Ausführung der Konstruktion erhält man im Schnitt eines Kreises mit einer Geraden oder einem anderen Kreis stets mehrere Punkte. Bei der Fortsetzung der Konstruktion bedarf es dann noch einer Angabe darüber, welcher der beiden beim $n$-ten Schritt erhaltenen Punkte bei der Fortsetzung der Konstruktion benutzt werden soll. Man kann fragen, bei welchen Konstruktionsaufgaben es stets gleichgültig ist, welchen der beim $n$-ten Schritt erzeugten Punkte man beim $(n+1)$-ten Konstruktionsschritt verwendet.}\footnote{Tietze~\cite[\S 4]{tietze1944} considered the problem mentioned in this section and noted that the intersection of a circle with a line or other circle contains several points, and one should specify which or the points obtained at the $n$th step of the construction should be used. One may ask for which constructions it does not matter which point is used for the next $(n+1)$th step.}
\end{quote}

Indeed, Tietze studied this question in several papers, starting from 1909~\cite{tietze1909,tietze1940,tietze1944}. He noted the necessity of ordering tests for standard constructions already in~\cite{tietze1909} amd improves the exposition in~\cite{tietze1940}. Then~\cite[pp.~228--229]{tietze1944} he considers also the question: which \emph{distances} can be constructed without ordering tests? Tietze notes that one can construct a distance that is $\sqrt{3}$ times longer that a given one (just consider the distance between the intersection points of two circles with centres $A$, $B$ and radius $AB$) but this is not the case, say,  for $\sqrt{3}+1$ instead of $\sqrt{3}$.

One could give up and consider the non-uniform setting only. As Manin~\cite[p.~209]{manin1963} puts it, ``we ignore how to choose the required point from the set of points obtained by the construction'' (\foreignlanguage{russian}{<<остаётся в стороне вопрос о $\langle\ldots\rangle$ выборе из построенной совокупности точек>>}).
Another approach is to replace straight-line programs by decision trees where tests appear as internal nodes. Engeler~\cite[p.~66]{engeler1968} lists~$8$~types of tests (equality, incidence, non-empty intersection of lines and curves, order on a line, equal distances between two pairs of points).\footnote{One can also include the orientation test that asks whether a triangle $ABC$ has positive or negative orientation. Without such a test, one cannot construct the point $(0,1)$ given the points $(0,0)$ and $(1,0)$ in a uniform setting.} Still none of these two approaches (decision trees or non-deterministic choice) is enough to save some classical constructions in a uniform setting as observed by~Baston and Bostock~\cite[p.~1020]{baston-bostock1990}. They noted that the construction from Mohr--Mascheroni theorem allows us to construct the intersection point of two intersecting lines $AB$ and $CD$ (given $A,B,C,D$) using only a compass. Each use of the compass increases the diameter of the current configuration at most by an $O(1)$-factor, and the intersection point can be arbitrarily far even if $A,B,C,D$ are close to each other, so there could be no \emph{a priori} bound on the number of steps. The necessity of an iterative process in the Mohr--Mascheroni theorem was earlier mentioned in another form by Dono Kijne~\cite[ch.~VIII, p.~99]{kijne1956}; he noted that this result depends on Archimedes' axiom.

To save the Mohr--Mascheroni construction, one may consider programs that allow loops. This was suggested, e.g., by Engeler~\cite{engeler1968}. Here we should specify what kind of data structures are allowed (e.g., whether we allow dynamic arrays of geometric objects or not).\footnote{For example, Huckenbeck~\cite{huckenbeck} considers the programs fixed number of variables called ``$\textit{GSM}_0$s with conditional jumps''. He does not specify the class of tests, but it is not clear whether, say, Mohr--Mascheroni construction can be implemented in a natural way in this framework.} In this way we encounter another problem, at least if we consider straightedge-only constructions in the \emph{rational} plane $\mathbb{Q}^2$ and allow using tests, and at the same time we do not bound the number of steps/objects. Baston and Bostock~\cite{baston-bostock1990} observed that having four different points $A,B,C,D\in\mathbb{Q}^2$ in a general position (no three points lie on a line, no two connecting lines are parallel), we can enumerate all (rational) points and therefore all rational lines. Then we can wait until a line parallel to $AB$ appears (we assume that we may test whether two given lines intersect or are parallel) and then use this parallel line to find the midpoint of $AB$. This construction does not look like a intuitively valid geometric construction and contradicts the belief that one cannot construct the midpoint using only a straightedge, see~\cite{baston-bostock1990} for details.

\section{Arbitrary Points}\label{sec:arbitrary}

Let us now consider the other (and probably more serious) reason why the notion of a derivable object differs from the intuitive notion of a constructible object. Recall the statement about angle trisection as stated by Tao~\cite{tao2011}: for some triangle $ABC$ the trisectors of angle $BAC$ are not derivable from $\{A,B,C\}$. (Tao uses the word ``constructible'', but we keep this name for the intuitive notion, following~\cite{baston-bostock1990}.) Tao interprets this statement as the impossibility of angle trisection with a compass and straightedge, and for a good reason.

On the other hand, the center of a circle is not derivable from the circle itself, for the obvious reason that no operation can be applied to enlarge the collection that consists only of the circle. Should we then say that the center of a given circle cannot be constructed by straightedge and compass? Probably not, since such a construction is well known from the high school. A similar situation happens with the construction of a bisector of a given angle (a configuration consisting of two lines and their intersection point).

Looking at the corresponding standard constructions, we notice that they involve another type of steps, ``choosing an arbitrary point'' (on the circle or elsewhere). But we cannot just add the operation ``add an arbitrary point'' to the list of allowed operations, since all points would become derivable.  So what are the ``arbitrary points'' that we are allowed to add?  Bieberbach~\cite[p.~21]{bieberbach1952} speaks about \foreignlanguage{german}{``Punkte, \"uber die keine Angaben affiner oder metrischer Art gemacht sind''}\footnote{points for which we do not have affine or metric information} and calls them \foreignlanguage{german}{``willk\"urliche Punkte''}(arbitrary points) --- but this hardly can be considered as a formal definition.

Tietze does not consider the problem of arbitrary points; in~\cite{tietze1944} he writes:
\begin{quote}
\foreignlanguage{german}{18. Was wir nicht in unsere Betrachtungen einbezogen haben, sind willkürlich gewählte Elemente, wie sie bei manchen Konstruktionen eine Rolle spielen. Diese Rolle ist aber nicht ganz so einfach, als es bisweilen dargestellt wird, weil beispielsweise eine willkürlich gewählte Gerade, je nach der getroffenen Wahl, mit einer gegebenen oder bereits konstruierten oder später zu konstruierenden Kreislinie (oder auch mit der Figur schon angehörenden oder später zu konstruierenden Geraden) Schnittpunkte haben kann oder nicht. $\langle\ldots\rangle$ Ob es einmal in einer Fortsetzung dieser Note zu einer Besprechung auch dieser Dinge kommen wird, muß offen bleiben.}\footnote{What we do not include in our analysis are arbitrary elements that are used in some constructions. Their role is not so simple, however, as it is sometimes thought, since, for example, an arbitrary line may (depending on the choice made) intersect or may not intersect some circle that is already constructed or will be constructed later; the same is true for other figures formed by existing or future lines. $\langle\ldots\rangle$ We may (or may not) return to this question in the continuation of this note.}
\end{quote}

Baston and Bostock~\cite{baston-bostock1990} do not even attempt to give a formal definition of a constructible point. However, they acknowledge that the use of ``arbitrary'' points in necessary for some constructions, and make some (rather vague, to be honest) remarks about that (p.~1018):
\begin{quote}
\emph{Result} 1. Given three distinct collinear points $A$ , $B$, and $C$ such that $B$ is the midpoint of the segment $AC$, then it is possible using a ruler alone to construct the midpoint $M$ of the segment $AB$.

Here we see that $L(\{A, B, C\}) = \{A, B, C\}$ [the left-hand side is the closure of points $A$, $B$, $C$] so, although the midpoint $M$ is meant to be constructible from the set $\{A, B, C\}$, it is not actually derivable. Of course the constructibility of $M$ depends on the (not unreasonable) ability to choose arbitrary$^1$ points not on the line $ABC$. Although the possibility of sets giving rise to constructible points which are not derivable may be somewhat disturbing, it can be of some comfort that such sets are probably quite limited. Fairly clear evidence will appear later which essentially indicates that if a set $E$ contains at least four points, not the vertices of a parallelogram and no three of them collinear, then the ability to choose arbitrary points (with reasonable properties) will not permit the construction with ruler alone of any point which is not $l$-derivable [=derivable using a ruler only]. In the case of ruler-and-compasses or compasses-only constructions, we suspect the situation is even simpler. Although we will not be presenting any real evidence to support the view, we do have a strong feeling that, whenever $E$ contains at least two points, there are no constructible points which are not then derivable. Thus our overall conclusion is that the distinction between constructibility and derivability arising from the use of arbitrary points is not very complex. Since, in addition, this aspect of constructibility does not actively cause the inconsistencies in the literature to be described presently, we will not pursue a more detailed analysis in this direction.
\end{quote}
The footnote ``${}^1$'' says ``The interested reader may wish to consult [4,p.~79] where an elementary approach to the use of arbitrary points can be found.'' The reference points to the book~\cite{kutuzov1950}; however,  the corresponding explanations (p.~46 of the Russian original version) are also far from being clear:
\begin{quote}
One often uses \emph{arbitrary} elements in the geometric constructions. The option of adding arbitrary elements to the current configuration needs to be restricted, and there restrictions are usually formulated as follows:

\textsc{Restriction A}. \emph{One may consider an arbitrary point of the plane outside the given line as constructed}.

\textsc{Restriction B}. \emph{One may consider an arbitrary point on a given line that differs from the already constructed points in this line, as constructed}.

The items A and B are essential for many construction problems. $\langle\ldots\rangle$

Of course, when proving the correctness of a construction that \emph{uses arbitrary elements} we should not use any special properties of these elements and should consider them as essentially arbitrary points.\footnote{\foreignlanguage{russian}{Часто пользуются при построениях \emph{произвольными элементами}. Возможность приобщения произвольных точек к числу данных или уже построенных нуждается в особых оговорках, которые обычно формулируются в виде следующих требований:}

\foreignlanguage{russian}{Т\,р\,е\,б\,о\,в\,а\,н\,и\,е\, A. \emph{Может считаться построенной произвольная точка плоскости вне данной прямой}.}

\foreignlanguage{russian}{Т\,р\,е\,б\,о\,в\,а\,н\,и\,е\, B. \emph{Может считаться построенной произвольная точка на данной прямой, не совпадающая ни с одной из уже построенных на прямой точек}.}

\foreignlanguage{russian}{Требования A и B являются существенными при решении многих задач. $\langle\ldots\rangle$}

\foreignlanguage{russian}{При доказательстве правильности того или иного построения \emph{с использованием произвольных элементов} мы не должны, конечно, опираться на какие-либо специальные свойства этих элементов, а должны существенным образом считать эти элементы произвольными.}}
\end{quote}

Another popular exposition~\cite[p.~18]{kac-ulam1968} explains the use of arbitrary elements (in the context of finding the center of a given circle by a straightedge only) as follows:

\begin{quote}
What is a construction by straightedge alone? It is a \emph{finite} succession of steps, each requiring that either a straight line be drawn, or a point of intersection of two lines or of a line and the given circle be found. A straight line can be drawn through two points chosen more or less arbitrarily. For example, a step may call for choosing two arbitrary points on the circumference of the circle and joining them by a straight line. Or it may be drawn through two points, one or both of which have been determined in a previous step in the construction as intersections of lines or a line and the circle. The succession of steps eventually must yield a point that can be proved to be the centre of our circle.  
\end{quote}

Both quotes speak about \emph{proofs} that the constructed point has the desired properties (though do not specify what kind of proofs they have in mind). This is an old theme that goes back to Hilbert. In his classical book on the foundations of geometry~\cite{hilbert1902} there is a chapter ``Geometrical constructions based upon the axioms I--V'' that speaks, e.g., about ``those problems in geometrical constructions, which may be solved by the assistance of only the axioms I--V'', but there are no exact definitions of what does it mean. 

Probably the most detailed exposition of the role of arbitrary points is given by Manin in~\cite{manin1963} (an encyclopedia of elementary mathematics addressed to advanced high school students and teachers):
\begin{quote}
It is convenient to describe the construction process inductively. We start with a finite set of points of the plane and want to obtain another finite set of points. The construction process adds some new points to the existing ones, and then selects the answer to our problem among the points constructed. The second (selection) part depends on the nature of the problem; now we are interested in the process of \,a\,d\,d\,i\,n\,g\, new points.

By a \emph{construction step} we mean the process of adding \emph{one} new point. To find this new point, we perform some operations. \emph{\,B\,y\ \ d\,e\,f\,i\,n\,i\,t\,i\,o\,n}, a compass and straightedge construction may use only the following operations (note that operations 1 and 2 can be used several times while operation 3 is used once during the construction step).

1. \emph{Drawing a line through two points from a current set of points}. (By the current set of points we mean the result of the preceding construction step or the initial set for the first step.)

2. \emph{Drawing a circle whose center is a point from the current set that contains some other point from the current set.} [Manin does not allow to draw a circle that has center $A$ and radius $BC$ where $A,B,C$ are already constructed, but this does not matter much.] $\langle\ldots\rangle$

3. \emph{Choosing one intersection points of the lines and circles constructed and adding this point to the current set.}

Instead of operation 3 where some \,s\,p\,e\,c\,i\,f\,i\,c\, point is added, we sometimes need to use the following operation

3a. \emph{Choosing an ``arbitrary'' point and adding it to the current set.} 

One should specify what do we mean by ``arbitrary'' point. By definition, it means that the point can be ``arbitrarily'' chosen either on some line segment, or arc of a circle, or in some part of the plane bounded by segments or arcs (this part can also be unbounded). When doing this, we may use only the lines and circles constructed at the current construction step, and the endpoints of the segments and arc should belong to the current set. We may assume, of course, that the ``arbitrary  point'' is not an endpoint of a segment/arc or belongs to the interior of the part of the plane where it is chosen.

\emph{A compass and straightedge construction is a finite sequence of steps described above}. $\langle\ldots\rangle$

Summarizing, let us assume that a finite set of points of the plane is given. We say that a point $A$ \emph{is constructible} (by compass and straightedge), if there exists a construction such that the resulting set of points contains $A$ for all possible intermediate ``arbitrary'' choices. A construction problem is \emph{solvable} if it requires to find a set of points where each points is constructible by compass and straightedge from the given points.%
\footnote{\foreignlanguage{russian}{Удобно описать процесс построения индуктивно. Мы начинаем с конечного числа точек на плоскости и хотим получить конечное число точек на плоскости; процесс построения состоит в том, что к уже имеющейся системе точек мы добавляем по известным правилам ещё некоторые, а затем отбираем среди всех получившихся точек те, которые доставляют решение нашей задачи. Вторая часть, конечно, определяется спецификой задачи; нас интересует сейчас, по каким правилам \,д\,о\,б\,а\,в\,л\,я\,ю\,т\,с\,я\, точки.}

\foreignlanguage{russian}{Назовём \emph{шагом построения} добавление \emph{одной} новой точки к уже имеющимся. Отыскание этой новой точки производится в результате проведения некоторых операций; \emph{\,п\,о\, \,о\,п\,р\,е\,д\,е\,л\,е\,н\,и\,ю,\,} в построении циркулем и линейной шаг может состоять только из следующих операций (причём операции $1$ и $2$ могут применяться несколько раз, а операция $3$~--- один раз).}

\foreignlanguage{russian}{1. \emph{Проведение прямой через пару точек имеющейся совокупности}. (Эта совокупность является результатом предыдущего шага или представляет собой первоначально заданную систему точек.)}

\foreignlanguage{russian}{2. \emph{Проведение окружности с центром в одной из точек имеющейся совокупности и проходящей через некоторую другую точку этой совокупности.}} $\langle\ldots\rangle$

\foreignlanguage{russian}{3. \emph{Выбор одной точки пересечения построенных прямых и окружностей между собой и добавление этой точки к имеющейся совокупности}.}

\foreignlanguage{russian}{Вместо операции $3$, которая представляет собой выбор \,о\,п\,р\,е\,д\,е\-\,л\,ё\,н\,н\,о\,й\, \,т\,о\,ч\,к\,и,\, иногда оказывается необходимым пользоваться операцией}

\foreignlanguage{russian}{3a. \emph{Выбор <<произвольной>> точки и добавление её к имеющейся совокупности}.}

\foreignlanguage{russian}{Следует уточнить здесь употребление слова <<произвольной>>. По определению это означает, что точку можно <<произвольно>> выбирать либо на некотором отрезке прямой, либо на дуге окружности, либо же в части плоскости, ограниченной отрезками или дугами, и, возможно, уходящей в бесконечность. При этом все фигурирующие прямые и окружности должны быть построены на данном шаге, а концы отрезков и дуг должны быть точками имеющейся совокупности. Можно считать, конечно, что <<произвольная>> точка не должна лежать на концах отрезков и дуг или на границе упомянутой части плоскости.}

\foreignlanguage{russian}{\emph{Построением с помощью циркуля и линейки называется последовательность, состоящая из конечного числа описанных шагов}.} $\langle\ldots\rangle$

\foreignlanguage{russian}{Итак, пусть задана некоторая конечная совокупность точек на плоскости; мы считаем, что точку  $A$ \emph{можно построить} (с помощью циркуля и линейки), если существует такое построение, что (независимо от промежуточных <<произвольных>> выборов точек!) система точек, полученная в результате этого построения, содержит точку $A$. Задачу на построение мы считаем \emph{разрешимой}, если совокупность точек, которые следует найти для решения этой задачи, состоит только из таких точек, которые можно построить с помощью циркуля и линейки. }}
\end{quote}
Note that, being literally understood, this definition makes no sense. Indeed, a ``construction'' that should exist is a finite sequence of steps as described, so it (if understood in a natural sense) determines the ``arbitrary'' points that were added during operations of type 3a. So we cannot add the quantifier ``for all possible intermediate choices''.

How can we modify the definitions to make them rigorous? One of the possibilities is to consider the construction as a strategy in some game with explicitly defined rules. We discuss this approach in the next section.

\section{Game Definition}\label{sec:game}

The natural interpretation of the ``arbitrary choice'' is that the choice is made by an adversary.\footnote{Huckenbeck~\cite{huckenbeck2} uses a similar approach; however, he considers a rather strange condition on the arbitrary point: it should be one of the values of a fixed multi-valued function when the argument is an already constructed point. He also does not considers arbitrary strategies in such a game, but only strategies defined by straight-line programs.} In other words, we consider a game with two players, Alice and Bob. We start with the non-uniform version of this game.

Let $E$ be some finite set of geometric objects (points, lines, and circles). To define which objects $x$ are constructible starting from $E$, consider the following full information game. The \emph{position} of the game is a finite set of geometric objects. The initial position is $E$. During the game, Alice and Bob alternate. They have full information about the current position. Alice makes some requests, and Bob fulfills these requests by adding some elements to the current position. Alice wins the game when $x$ becomes an element of the current position. The number of moves is unbounded, so it is possible that the game is infinite (if $x$ never appears in the current position); in this case Alice does not win the game.

Here are possible request types.
\begin{itemize}
\item Alice may ask Bob to add to the current position some straight line that goes through two different points from the current position.
\item Alice may ask Bob to add to the current position a circle with center $A$ and radius $BC$, if $A,B,C$ are points from the current position.\footnote{We allow both $B$ and $C$ be different from $A$, as it is usually done, but this does not matter much.}
\item Alice may ask Bob to add to the current position one or two points that form the intersection of two different objects (lines or circles) that already belong to the current position.
\end{itemize}
If we stop here, we get exactly the notion of derivable points, though in a strange form of a ``game'' where Bob has no choice. To take the ``arbitrary'' points into account, we add one more operation:
\begin{itemize}
\item Alice specifies an open subset of the plane (say, an open circle), and Bob adds some point of this subset to the current position.
\end{itemize}
The game (depending on $E$ and $x$) is now described, and the point $x$ is \emph{constructible} from $E$ if Alice has a winning strategy in this game.

Let us comment on the last operation. 

(1)~Note that Alice cannot (directly) force Bob to choose some point on a line or on a circle, and this is often needed in the standard geometric constructions. But this is inessential since Alice can achieve this goal in several steps. First she asks to add points on both sides of the line or circle (selecting two small open sets on both sides in such a way that every interval with endpoints in these open sets intersects the line or circle), then asks to connect these points by a line, and then asks to add the intersection point of this new line and the original one. 

(2)~On the other hand, according to our rules, Alice can specify with arbitrarily high precision where the new point should be (by choosing a small open set). A weaker (for Alice) option would be to allow her to choose a connected component of the complement of the union of all objects in the current position. Then Bob should add some point of this component to the current position. (Recall that Alice and Bob have full information about the position.)

\begin{proposition}
This restriction does not change the notion of a constructible point.
\end{proposition}

\begin{proof}
Idea: Using the weaker option, Alice may force Bob to put enough points to make the set of derivable points dense, and then use the first three options to get a point in an arbitrary open set.  

Let us explain the details. First, she asks for an arbitrary point $A$, then for a point $B$ that differs from $A$, then for line $AB$, then for a point $C$ outside line $AB$ (thus getting the triangle $ABC$), then for the sides of this triangle, and then for a point $D$ inside the triangle. (All this is allowed in the restricted version.)
\begin{center}
\includegraphics{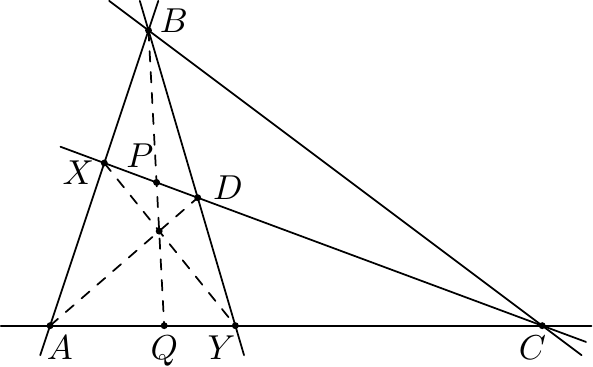}
\end{center}
Now the points $P$ and $Q$ obtained as shown in the picture are derivable (after the projective transformation that moves $B$ and $C$ to infinity, the points $P$ and $Q$ become the midpoints of $XD$ and $AY$). Repeating this construction, we get a dense set of derivable points on intervals $XD$ and $AY$, then the dense set of derivable points in the quadrangle $AXDY$, and then in the entire plane.

Now, instead of asking Bob for a point in some open set $U$, Alice may force him to include one of the derivable points (from the dense set discussed above) that is in~$U$.  
\end{proof}

This definition of constructibility turns out to be equivalent to the negative definition suggested by Akopyan and Fedorov~\cite{akopyan-fedorov2017}. They define non-constructibility as follows: an object $x$ is \emph{non-constructible} from a finite set $E$ of objects if there exists a set $E'\supset E$ that is closed under the operations of adding points, lines, and circles (contains all objects derivable from $E'$), contains an everywhere dense set of points, but does not contain~$x$.

\begin{proposition}[Akopyan--Fedorov]
This negative definition is equivalent to the game-theoretic definition given above.
\end{proposition}

\begin{proof}
The equivalence is essentially proven as~\cite[Proposition 15, p.~9]{akopyan-fedorov2017}, but Akopyan and Fedorov avoided stating explicitly the game-theoretic definition and spoke about ``algorithms'' instead (without an exact definition).

Assume that $x$ is non-constructible from $E$ according to the negative definition. Then Bob can prevent Alice from winning by always choosing points from $E'$ when Alice asks for a point in an open set. Since $E'$ is dense, these points are enough. If Bob follows this strategy, then the current position will always be a subset of $E'$ and therefore will never contain $x$.

On the other hand, assume that $x$ in not constructible from $E$ in the sense of the positive definition. Consider the following strategy for Alice. She takes some triangle $abc$ and point $d$ inside it and asks Bob to add points $A,B,C,D$ that belong to some small neighborhoods of $a,b,c,d$ respectively. The size of these neighborhoods guarantees that $ABC$ is a triangle and $D$ is a point inside $ABC$. There are two cases:
\begin{itemize}
\item for every choice of Bob Alice has a winning strategy in the remaining game;
\item there are some points $A,B,C,D$ such that Alice does not have a winning strategy in the remaining game.
\end{itemize}
In the first case Alice has a winning strategy in the entire game and $x$ is constructible in the sense of the positive definition. In the second case we consider the set $E'$ of all objects derivable from $E\cup\{A,B,C,D\}$. As we have seen in the proof of the previous proposition, this set is dense. Therefore, $x$ is non-constructible in the sense of the negative definition.
\end{proof}

The advantage of the game definition is that it can be reasonably extended to the uniform case. For the uniform case the game is no more a full-information game. Alice sees only the names (and types) of geometric objects in $E$, and assigns names to new objects produced by Bob. One should agree also how Alice can get information about the configuration and how she can specify the connected component when asking Bob for a point in this component. For example, we may assume that Alice has access to the list of all connected components and the full topological information about the structure they form, as well as the places of objects from $E$ in this structure. Then Alice may choose some component and request a point from it. To win, Alice also needs to specify the name of the required object $x$. After we agree on the details of the game, we may define construction algorithms as computable strategies for such a game. (Note that in this version Alice deals only with finite objects.)

\section{Formal Definitions Are Important}\label{sec:error}

In fact, the absence of formal definitions and exact statements is more dangerous than one could think. It turned out that some classical and well known arguments contain a serious gap that cannot be filled without changing the argument. This happened with a proof (attributed to Hilbert in~\cite{cauer1913}) that one cannot find the center of a given circle using only a straightedge. It is reproduced in many popular books (see, e.g.,~\cite{rademacher1933,courant1996,kac-ulam1968}) and all the arguments (at least in the four sources mentioned above) have the same gap. They all go as follows~\cite[p.~18]{kac-ulam1968}: 
\begin{quotation}
Let the construction be performed in a plane $P_1$ and imaging a transformation or mapping $T$ of the plane $P_1$ into another plane $P_2$ such that:

(a) straight lines in $P_1$ transform into straight lines in $P_2$ $\langle\ldots\rangle$

(b) The circumference $C$ of our circle is transformed into a circumference $T(C)$ for some circle in $P_2$.

As the steps called for in the construction are being performed in $P_1$, they are being faithfully copied in $P_2$. Thus when the construction in $P_1$ terminates in the centre $O$ of $C$, the ``image'' construction \emph{must} terminate in the centre $T(O)$ of the circle $T(C)$.

Therefore if one can exhibit a transformation $T$ satisfying (a) and (b), but such that $T(O)$ \emph{is not} the centre of $T(C)$, then the impossibility of constructing the centre of a circle by ruler alone will be demonstrated.
\end{quotation}
Such a transformation indeed exists, but the argument in the last paragraph has a gap. If we understand the notion of construction in a non-uniform way and require that the point was among the points constructed,  the argument does not work since the center of $T(C)$ could be the image of some other constructed point. If we use some kind of the uniform definition and allow tests, then these tests can give different results in $P_1$ and $P_2$ (the projective transformation used to map $P_1$ into $P_2$ does not preserve the ordering), so there is no reason to expect that the construction is ``faithfully copied''. And a uniform definition that does not allow tests and still is reasonable, is hard to imagine (and not given in the book). Note also that some lines that intersect in $P_1$, can become parallel in $P_2$.

It is easy to correct the argument and make it work for the definition of constructibility given above (using the fact that there are many projective mappings that preserve the circle), but still one can say without much exaggeration that the first correct proof of this impossibility result appeared only in~\cite{akopyan-fedorov2017}. One can add also that the stronger result about two circles that was claimed by Cauer~\cite{cauer1913} and reproduced with a similar proof in~\cite{rademacher1933}, turned out to be plainly false as shown in~\cite{akopyan-fedorov2017}, and the problems in the proof were noted already by Gram~\cite{gram1956}. It is not clear why Gram did not question the validity of the classical proof for one circle, since the argument is the same. Gram did not try to give a rigorous definition of the notion of a geometric construction, speaking instead about constructions in the ``ordered plane'' and referring to Bieberbach's book~\cite{bieberbach1952} that also has no formal definitions.

The weak version of Cauer's result saying that for some pairs of circles one cannot construct their centers, can be saved and proven for the definition of constructibility discussed above~(see \cite{akopyan-fedorov2017} and the popular exposition in~\cite{shen2018}).

It would be interesting to reconsider the other results claimed about geometric constructions (for example, in~\cite{hilbert1902,tietze1909,tietze1940,tietze1944}) to see whether the proofs work for some clearly defined notion of a geometric construction. Note that in some cases (e.g., for Tietze's results) some definition of the geometric construction for the uniform case is needed (and the negative definition is not enough).

\textbf{Acknowledgements.}
The authors thank Sergey Markelov, Arseny Akopyan, Roman Fedorov, Rupert H\"olzl and their colleagues at Moscow State University and LIRMM (Montpellier) for interesting discussions, and the referees for their comments and suggestions. The work was supported by ANR-15-CE40-0016-01 RaCAF and RFBR 16-01-00362 grants.


\begin{thebibliography}{9}

\bibitem{akopyan-fedorov2017}
Akopyan, A., Fedorov, R., \emph{Two circles and only a straightedge}, \url{https://arxiv.org/abs/1709.02562} (2017)

\bibitem{baston-bostock1990}
Baston, V.J., Bostock, F.A., On the impossibility of ruler-only constructions, \emph{Proceedings of the American Mathematical Society}, \textbf{110} (4), December 1990

\bibitem{bieberbach1952} Bieberbach, L.,\foreignlanguage{german}{\emph{Theorie der geometrischen Konstruktionen}}, Springer Basel AG (1952), \url{DOI 10.1007/978-3-0348-6910-2}

\bibitem{borel1912}
Borel, E., \foreignlanguage{french}{Le calcul des int\'egrales d\'efinies, \emph{Journal de Math\'ematiques pures et appliqu\'ees}}, ser. 6, \textbf{8}(2), 159--210 (1912)

\bibitem{cauer1913}
Cauer, D.,  \"Uber die Konstruktion des Mittelpunktes eines Kreises mit dem Lineal allein. \emph{Mathematische Annallen}, \textbf{73}. S.~90--94. A correction (Berichtigung): \textbf{74}, S.~462--464 (1913)

\bibitem{courant1996}
Courant, R., Robbins, H.,  revised by Stewart, I., \emph{What is Mathematics? An elementary approach to ideas and methods}. Oxford University Press (1996)

\bibitem{durand-lose}
J\'er\^ome Durand-Lose, Ways to compute in Euclidean Frameworks, \emph{Unconventional computation and natural computation, 16th international conference}, UCNC~2017. Proceedings. Lecture Notes in Computer Science, v.10240, 8--28 (2017)

\bibitem{engeler1968}
Engeler, E.,  Remarks on the theory of geometrical constructions, \emph{The Syntax and Semantics of Infinitary Languages}, edited by Jon Barwise, Springer (Lecture Notes in Mathematics, \textbf{72}), pp.~64--76 (1968)

\bibitem{gram1956}
Gram, C., A remark on the construction of the centre of a circle by means of the ruler, \emph{Math. Scand.}, \textbf{4}, 157--160 (1956)

\bibitem{hilbert1902}
Hilbert, D., \emph{The Foundations of Geometry}, authorized translation by E.J.~Townsend, Ph.D., University of Illinois (1902). Reprint edition: The Open Court Publishing Company, La Salle, Illinois (1950). Available at \url{https://archive.org/details/thefoundationsof17384gut}

\bibitem{huckenbeck}
Huckenbeck, U., Euclidean geometry in terms of automata theory, \emph{Theoretical Computer Science}, \textbf{68}(1), 71--87 (16 October 1989) 

\bibitem{huckenbeck2}
Huckenbeck, U., On geometric automata which can nondeterministically choose auxiliary points, \emph{Informatique th\'eorique et applications}, \textbf{24} (5), 471--487, \url{http://www.numdam.org/item?id=ITA_1990__24_5_471_0} (1990)

\bibitem{kac-ulam1968}
Kac, M.,  Ulam, S.M., \emph{Mathematic and Logic}, Dover publications, 1992, copyright: Encyclopaedia Britannica (1968)

\bibitem{kijne1956}
Kijne, D., \emph{Plane construction field theory}, Ph.D. thesis, promotor H.~Freudenthal, van Gorcum \& Co., N.V., G.A.Hak \& Dr.~H.J.~Prakke (28 May 1956)

\bibitem{kutuzov1950} 
Kutuzov,  B.V., \emph{Studies in mathematics}, vol.~IV, Geometry (L.I.Gordon and E.S.Shater, transl.) School Mathematics Study Group, University of Chicago (1960). Russian original:  \foreignlanguage{russian}{Кутузов Б.В.. \emph{Геометрия. Пособие для учительских и педагогических институтов}. Москва, Учпедгиз} (1950)

\bibitem{manin1963} 
Manin, Yu.  (\foreignlanguage{russian}{Манин, Ю.И.}), On the decidability of geometric construction problems using compass and straightedge [Russian] (\foreignlanguage{russian}{О разрешимости задач на построение с помощью циркуля и линейки}). \emph{Encyclopedia of Elementary Mathematics, Volume IV, Geometry}, Moscow  (\foreignlanguage{russian}{Энциклопедия элементарной математики. Том IV. Геометрия. Москва: государственное издательство физико-математической литературы, 1963. Под редакцией В.\,Г.\,Болтянского и И.\,М.\,Яглома}),  205--227 (1963)

\bibitem{martin1998}
Martin, G.E., \emph{Geometric Constructions}, Springer (1998)

\bibitem{rademacher1933}
Rademacher, H.,  Toeplitz O., \emph{\foreignlanguage{german}{Von Zahlen und Figuren. Proben mathematischen Denkens f\"ur Liebhaber der Mathematik ausgew\"ahlt und dargestellt.} Zweite Auflage. Springer-Verlag, Berlin, Heidelberg}, 1933.

\bibitem{schreiber1975}
Schreiber, P., \emph{Theorie der geometrischen Konstruktionen}, VEB Deutscher Verlag der Wissenschaften, Berlin (1975)

\bibitem{shen2018}
Shen, A., \emph{Hilbert's Error?}, \url{https://arxiv.org/abs/1801.04742} (2018)

\bibitem{tao2011}
Tao, T., \emph{A geometric proof of the inpossibility of angle trisection},\\
{\small\url{https://terrytao.wordpress.com/2011/08/10/a-geometric-proof-of-the}\\ \url{-impossibility-of-angle-trisection-by-straightedge-and-compass/}} 

\bibitem{tietze1909}
Tietze, H., \foreignlanguage{german}{\"Uber dir Konstruierbarkeit mit Lineal und Zirkel, \emph{Sitzungsberichte der Kaiserlichen Akademie der Wissenschaften}}, Abt.~IIa, 735--757 (1909). A short resume: 
    \foreignlanguage{german}{%
         \emph{Anzeiger der Kaiserlichen Akademie der Wissenschaften}, 
         Mathematisch-Naturwissenschaftliche Klasse, XLVI. Jahrgang%
     }%
     , Nr.~XIII, 208--210 (21. Mai 1909),
  \url{https://www.biodiversitylibrary.org/item/93371}

\bibitem{tietze1940}
Tietze, H., \foreignlanguage{german}{\"Uber die mit Lineal und Zirkel und die mit dem rechten Zeichenwinkel l\"osbaren Konstruktionaufgaben I}, \emph{Mathematische Zeitschrift}, \textbf{46}, 190--203 (1940), \url{http://www.digizeitschriften.de/dms/img/?PID=GDZPPN002379074}

\bibitem{tietze1944}
Tietze, H., \foreignlanguage{german}{Zur Analyse der Lineal- und Zirkelkonstruktionen. I. \emph{Sitzungsberichte der mathematisch-naturwissenschaftlichen Abteilung der Bayrischen Akademie der Wissenschaften zu M\"unchen, 1944,  Heft III, Sitzungen Oktober--Dezember},  209--231, M\"unchen} (1947). Available at \url{http://publikationen.badw.de/003900992.pdf} 

\bibitem{wantzel1837}
Wantzel, M.L., \foreignlanguage{french}{Recherches sur les moyens de reconna\^\i tre si un probl\'eme de G\'eom\'etrie peut se r\'esoudre avec la r\`egle et le compas, \emph{Journal de math\'ematiques pures et appliqu\'ees}, 1re s\'erie}, \textbf{2}, 366--372 (1837)
 
\end{thebibliography}
\end{document}